\documentclass[10pt,a4paper]{article}
\usepackage{fullpage}
\usepackage{amsfonts,amsmath,amssymb}
\usepackage{graphicx}
\usepackage{sectsty}
\newtheorem{theorem}{Theorem}[section]
\newtheorem{lemma}[theorem]{Lemma}

\newtheorem{proposition}[theorem]{Proposition}

\newenvironment{proof}[1][Proof]{\begin{trivlist}
\item[\hskip \labelsep {\bfseries #1}]}{\end{trivlist}}
\numberwithin{equation}{section}

\newenvironment{acknowledgement}[1][Acknowledgement
]{\begin{trivlist} \item[\hskip \labelsep {\bfseries
#1}]}{\end{trivlist}}
\newenvironment{keywords}[1][keywords]{\begin{trivlist} \item[\hskip \labelsep {\bfseries
#1}]}{\end{trivlist}}
\begin{document}
\title {Counting the Resonances in High and Even Dimensional Obstacle Scattering}
\author{Lung-Hui Chen$^1$}\maketitle\footnotetext[1]{Department of
Mathematics, National Chung Cheng University, 168 University Rd.,
Min-Hsiung, Chia-Yi County 621, Taiwan. Email:
lhchen@math.ccu.edu.tw; mr.lunghuichen@gmail.com. Fax:
886-5-2720497.  The author is supported by NSC Grant
97-2115-M194-010-MY2.}

\begin{abstract}
In this paper, we give a polynomial lower bound for the resonances
of $-\Delta$ perturbed by an obstacle in even-dimensional
Euclidean spaces, $n\geq4$. The proof is based on a Poisson
Summation Formula which comes from the Hadamard factorization
theorem in the open upper complex plane. We take advantage of the
singularity of regularized wave trace to give the pole/resonance
counting function over the principal branch of logarithmic plane a
lower bound.
\end{abstract}

\section{Introduction}
Let $H$ be an embedded hypersurface in $\mathbb{R}^n$ such that
\begin{equation}
\mathbb{R}^n\setminus H=\Omega\cup\mathcal{O},\mbox{ with
}\overline{\mathcal{O}}\mbox{ compact and }\overline{\Omega}\mbox{
connected },n\geq4,
\end{equation}where both $\mathcal{O}$ and $\Omega$ are open. We
refer to $\overline{\mathcal{O}}$ as the obstacle and $\Omega$ as
the exterior.

\par
In the exterior domain $\Omega$, we consider the resonance theory of the
differential operator
\begin{equation}\label{1.1}
P:=-\Delta:L^2(\Omega)\rightarrow L^2(\Omega)
\end{equation}
satisfying the boundary condition either
\begin{equation}
u|_H=0
\end{equation}
or
\begin{equation}
(\frac{\partial}{\partial
n}+\gamma)u|_H=0,\gamma\in\mathcal{C}^\infty(H).
\end{equation}
Let us denote
\begin{equation}
\Lambda_{\frac{1}{2}}:=\{\lambda\neq0|0<\arg\lambda<\pi\}.
\end{equation}
It is well-known from spectral analysis that the resolvent
operator $(P-\lambda^2)^{-1}:L^2(\Omega)\rightarrow H^2(\Omega)$
is bounded in $\Lambda_{\frac{1}{2}}$ except a finite discrete set $\{\mu_1,\cdots,\mu_m\}$ such that
$\{\mu_1^2,\cdots,\mu_m^2\}$ are the eigenvalues of $P$. As a special case of the
black box formalism in Zworski and Sj\"{o}strand \cite{Sjostrand
and Zworski1991}, $(P-\lambda^2)^{-1}$ can be meromorphically
extended from $\Lambda_{\frac{1}{2}}$ to
$\Lambda:=\{\lambda\neq0|-\infty<\arg\lambda<\infty\}$, the
logarithmic plane, as an operator
\begin{equation}
R(\lambda):=(P-\lambda^2)^{-1}:L^2_{\rm comp}(\Omega)\rightarrow
H^2_{\rm loc}(\Omega)
\end{equation}
with poles of finite rank. All such meromorphic poles in $\Lambda$
are called resolvent resonances in mathematical physics
literature. Classically, we have Lax-Phillips theory for the
meromorphic extension theory of $R(\lambda)$ and of the relative
scattering matrix $S(\lambda)$ induced by scatterer $\mathcal{O}$.
See Lax and Phillips \cite{Lax,Lax2}. For a more complicated
boundary setting, we refer to Shenk and Thoe \cite{Shenk}.

\par
A question for the meromorphically extended $R(\lambda)$ would be: how many
resonances approximately inside a disc of radius $r$ in the
principal sheet of $\Lambda$? Let us
array all of the resonances as a sequence
$\{\lambda_1,\lambda_2,\cdots\}$ repeated according to their
multiplicity. Let us define
\begin{equation}
\Lambda_1:=\{\lambda\neq0|-\frac{\pi}{2}\leq\arg\lambda<\frac{3\pi}{2}\},
\end{equation}
where we set the branch cut to be negative
imaginary axis. In this paper, we study the resonance counting
function
\begin{equation}\label{1.10}
N_1(r):=\sharp\{\lambda\in\Lambda_1|\mbox{ resonances of }
R(\lambda), 0<|\lambda|\leq r\},\hspace{2pt}r>1.
\end{equation}
\par
From the literature, we define
\begin{equation}
N(r,a):=\sharp\{\lambda\in\Lambda|\mbox{ resonances of }
R(\lambda), 0<|\lambda|\leq r,\hspace{2pt}|\arg\lambda|\leq
a\},\hspace{2pt}r,a>1.\label{counting}
\end{equation}
The optimal upper bound of this counting function in Vodev's
formalism is found in \cite{Vodev,Vodev2} from which we recall for
$n\geq2$,
\begin{equation}
N(r,a)\leq Ca(r^n+(\log a)^n),\mbox{ where }C\mbox{ is a
constant},\hspace{2pt}r,a>1.\label{119}
\end{equation}
For $a<\frac{\pi}{2}$, an upper bound of $N(r,a)$ also
follows from Sj\"{o}strand and Zworski \cite{Sjostrand and
Zworski1991}. That is
\begin{equation}\label{110}
N(r,a)=\mathcal{O}(r^n).
\end{equation}
Consequently,~(\ref{119}) and~(\ref{110}) give a polynomial upper
bound on the resonance counting function $N_1(r)$ in
$\Lambda_{1}$ for $n\geq2$.

\par
To count $N_1(r)$, we begin our analysis with the domain
\begin{equation}
\Lambda_{\pi/2}:=\{\lambda\neq0||\arg\lambda|<\pi/2,\hspace{2pt}|\arg\lambda-\pi|<\pi/2\}.
\end{equation}
Defining the scattering determinant
\begin{equation}\label{112}
s(\lambda):=\det S(\lambda),
\end{equation}
where $S(\lambda)$ is the scattering matrix in sense of Zworski
\cite{Zworski3}. The poles of $s(\lambda)$ coincide with
$S(\lambda)$'s. We recall the functional identities for even
dimensional scattering matrix in Shenk and Thoe
\cite[p.468]{Shenk}, using the same notation for scattering
matrix,
\begin{equation}
S(\lambda)^{-1}=S(\overline{\lambda})^\ast,\mbox{ when
}\lambda\mbox{ is not a pole of }S(\lambda)\mbox{ or
}S(\overline{\lambda})^\ast;\label{115}
\end{equation}
\begin{equation}
S(\overline{\lambda})^\ast=2I-S(e^{\pi i}\lambda),\mbox{ when
}\lambda\in\Lambda.\label{1.11}
\end{equation}
We define $m_\mu(R)$ to be the multiplicity of $S(\lambda)$ near
the pole $\lambda=\mu$. According to~(\ref{115}), we know
$\lambda=\overline{\mu}$ is a zero of $S(\lambda)$. There exist
finitely many eigenvalues, say, $\mu_1^2,\cdots,\mu_m^2$, repeated
according to their multiplicity, such that $\{\mu_j\}_{j=1}^m$
appear as the poles of $S(\lambda)$ in $\Lambda_{\frac{1}{2}}$.
The result~(\ref{110}) is sufficient to give the representation,
as a special case in Zworski \cite[(2.3)]{Zworski3},
\begin{equation}\label{1.15}
s(\lambda)=e^{g_{\pi/2}(\lambda)}\frac{P_{\pi/2}(-\lambda)}{P_{\pi/2}(\lambda)},\mbox{ where
}\lambda\in\Lambda_{\pi/2}\cap\{\Re\lambda>0\}.
\end{equation}
In general, $g_{\pi/2}$ is a symbol on $\mathbb{R}$ such that
\begin{equation}
|\partial_\lambda^kg_{\pi/2}(\lambda)|\leq
C_{k,\epsilon}(1+|\lambda|)^{n+\epsilon-k},\forall\epsilon>0,
\end{equation}
and there exists an $m_0$ such that
\begin{equation}\label{1.18}
P_{\pi/2}(\lambda):=\prod_{\{\mu\in\Lambda_{\pi/2}\setminus\mathbb{R}\}}E(\frac{\lambda}{\mu},m_0)^{m_\mu(R)}
\end{equation}
where
\begin{equation}\label{11.9}
E(z,p):=(1-z)\exp(1+\cdots+\frac{z^{p}}{p}).
\end{equation}
In this paper, we set
\begin{equation}
g_{\pi/2}(-\lambda)=-g_{\pi/2}(\lambda),\lambda>0.\label{111}
\end{equation}
We compare this construction to Zworski's in \cite[p.3]{Zworski3}.
Combining~(\ref{111}) with~(\ref{1.15}), the fact that
$S(\lambda)$ is unitary on $\lambda>0$ implies
\begin{equation}\label{1.2}
|s(\lambda)|=1,\forall\lambda\in0i+\mathbb{R}.
\end{equation}

\par
Now, we define
\begin{equation}
\Lambda_{\pi/2}^\ast:=\Lambda_{\pi/2}\cup e^{i\frac{\pi}{2}}\mathbb{R}.
\end{equation}
We describe the analytic behavior of $s(\lambda)$ on
$e^{i\frac{\pi}{2}}\mathbb{R}$ by the odd reflection defined
in~(\ref{111}) as follows.
\begin{lemma}
$s(\lambda)$, defined through~(\ref{111}), extends from
$\{\lambda||\arg\lambda|\leq\pi/2\}$ to $\mathbb{C}$ such that it
is analytic in the neighborhood along
$e^{i\frac{\pi}{2}}\mathbb{R}$ outside $\{\mu_1,\cdots,\mu_m\}$.
\end{lemma}
\begin{proof}
Let $V$ be a neighborhood on the imaginary axis
$e^{i\frac{\pi}{2}}\mathbb{R}$ outside $\{\mu_1,\cdots,\mu_m\}$.
Let $V:=V^-\cup e^{i\frac{\pi}{2}}\mathbb{R}\cup V^+$, where
$V^{\pm}$ is in the right/left sheet of $\Lambda_{\pi/2}$.

\par
We claim that $g_{\pi/2}(\lambda)$ extends to be purely real on
the imaginary. From~(\ref{1.15}) and~(\ref{115}),
\begin{equation}
|s(\lambda)|=|e^{\Re g_{\pi/2}+i\Im g_{\pi/2}}|=1,\forall
\lambda>0.
\end{equation}
Hence, $ g_{\pi/2}$ is purely imaginary on real by~(\ref{111}).
~(\ref{111}) says
\begin{equation}\label{1.24}
g_{\pi/2}(-\lambda)=-g_{\pi/2}(\lambda),\forall\lambda\in
0i+\mathbb{R}.
\end{equation}
Hence, using identity theorem,~(\ref{1.24}) implies
\begin{equation}\label{1.29}
g_{\pi/2}(-\lambda)=\overline{g_{\pi/2}(\overline{\lambda})},\forall\lambda\in\Lambda_{\pi/2}.
\end{equation}
We extend~(\ref{1.29}) to $e^{i\frac{\pi}{2}}\mathbb{R}$ by
observing
\begin{equation}\label{1.26}
\lim_{\Re\lambda\rightarrow0^-}g_{\pi/2}(-\lambda)=\lim_{\Re\lambda\rightarrow0^+}\overline{g_{\pi/2}(\overline{\lambda})}
=\overline{\lim_{\Re\lambda\rightarrow0^+}g_{\pi/2}(\overline{\lambda})}
=\overline{g_{\pi/2}(-\lambda)}.
\end{equation}
Hence, the claim is proved.

It is obvious that, $\forall\lambda\in
e^{i\frac{\pi}{2}}\mathbb{R}\cup V^+$,
\begin{equation}\label{1.30}
g_{\pi/2}(\overline{\lambda})\rightarrow
g_{\pi/2}(-\lambda),\mbox{ as }\Re\lambda\rightarrow0^+.
\end{equation}
On the other hand, in $V^-$, using~(\ref{1.29}) and~(\ref{1.26}),
\begin{equation}\label{1.31}
g_{\pi/2}(-\lambda)=\overline{g_{\pi/2}(\overline{\lambda})}\rightarrow\overline{g_{\pi/2}(-\lambda)}=g_{\pi/2}(-\lambda),\mbox{
as }\Re\lambda\rightarrow0^-.
\end{equation}
Hence, for such a $\lambda$,~(\ref{1.30}) and~(\ref{1.31})
combines to give
\begin{equation}
\lim_{\Re\lambda\rightarrow
0^+}g_{\pi/2}(\overline{\lambda})=\lim_{\Re\lambda\rightarrow
0^-}g_{\pi/2}(-\lambda).
\end{equation}
Therefore, the value of $s(\lambda)$ meet from both sides up to
$e^{i\frac{\pi}{2}}\mathbb{R}$ outside $\{\mu_1,\cdots,\mu_m\}$
and $s(\lambda)$ is analytic in $V$ by Schwarz reflection. See
Lang \cite[p.298]{Lang}. $\Box$
\end{proof}
As far as resonances counting in concerned, $s(\lambda)$ defined
through~(\ref{111}) has as much as resonances in
$\Lambda_{\pi/2}^\ast$ as the original scattering determinant in
$\Lambda_1$. Hence, we do not differentiate these two scattering
determinant here.
\par
Let us define the order of a function that is holomorphic in
$\Lambda_{\frac{1}{2}}$ to be the greatest lower bound of numbers
$\nu$ for which
\begin{equation}
\limsup_{r\rightarrow\infty}\frac{\ln|f(re^{i\theta})|}{r^\nu}=0,
\end{equation}
uniformly in $\theta$, $0<\theta<\pi$. See Levin \cite{Levin}. We
observe that
\begin{lemma}
$s(\lambda)\prod_{j=1}^m\frac{\lambda-\mu_j}{\lambda-\overline{\mu}_j}$
is holomorphic in $\Lambda_{\frac{1}{2}}$ with zeroes
$\{\overline{\lambda}_j\}$ there such that $\{\lambda_j\}$ are the
poles of $s(\lambda)$ in $\Lambda_{\frac{\pi}{2}}^\ast$.
$s(\lambda)\prod_{j=1}^m\frac{\lambda-\mu_j}{\lambda-\overline{\mu}_j}$
is a holomorphic function of integral order at most $n$ in
$\Lambda_{\frac{1}{2}}$.
\end{lemma}
The proof is a direct consequence due to the estimates of Vodev
\cite{Vodev,Vodev2} and Sj\"{o}strand and Zworski \cite{Sjostrand
and Zworski1991}. Given such a function
$s(\lambda)\prod_{j=1}^m\frac{\lambda-\mu_j}{\lambda-\overline{\mu}_j}$
with zeroes $\overline{\lambda}_j$ of growth order of at most $n$
in $\Lambda_{\frac{1}{2}}$, we can apply Govorov's theorem
\cite{Govorov}, also Levin \cite[Appendix VIII]{Levin}, which
serves as the Hadamard factorization theorem over
$\Lambda_{\frac{1}{2}}$. Govorov's theorem allows us to write down
 a Hadamard factorization theorem for an integral function of finite order
 in terms of its zeroes in an angle.

\par
In particular, we define
\begin{equation}
P(\lambda):=\prod_{\{\mu\in\{\lambda_j\}\}}E(\frac{\lambda}{\mu},n)^{m_\mu(R)}\mbox{
and
}\overline{P}(\lambda):=\prod_{\{\mu\in\{\lambda_j\}\}}E(\frac{\lambda}{\overline{\mu}},n)^{m_\mu(R)}
\end{equation}
where
\begin{equation}
E(z,p):=(1-z)\exp(1+\cdots+\frac{z^{p}}{p}).
\end{equation}
$\overline{P}(\lambda)$ is an entire function of finite order at
most $n$ in $\Lambda_{\frac{1}{2}}$. The conjugate is taken as in
$\mathbb{C}$. The convergence of this Weierstrass product is
guaranteed by~(\ref{119}) and~(\ref{110}). Then, using Govorov's
theorem,
\begin{equation}
s(\lambda)\prod_{j=1}^m\frac{\lambda-\mu_j}{\lambda-\overline{\mu}_j}=\exp\{i(a_0+a_1\lambda+\cdots+a_{n}\lambda^{n})\}
\exp\{\frac{1}{\pi
i}\int_{-\infty}^\infty(\frac{t\lambda+1}{t^2+1})^{n+1}\frac{d\Sigma(t)}{t-\lambda}\}
\frac{\overline{P}(\lambda)}{P(\lambda)},\hspace{2pt}\lambda\in\Lambda_{\frac{1}{2}},
\label{1.19}
\end{equation}
where $n$ is the order of the
$s(\lambda)\prod_{j=1}^m\frac{\lambda-\mu_j}{\lambda-\overline{\mu}_j}$,
$\{a_k\}$ are real constants, $\{\overline{\lambda}_j\}$ are the
zeroes of
$s(\lambda)\prod_{j=1}^m\frac{\lambda-\mu_j}{\lambda-\overline{\mu}_j}$
in $\Lambda_{\frac{1}{2}}$ and, most important of all,
\begin{equation}
\Sigma(t):=\lim_{y\rightarrow0^+}\int_0^t\ln|s(x+iy)\prod_{j=1}^m\frac{x+iy-\mu_j}{x+iy-\overline{\mu}_j}|dx.
\end{equation}
Since $n\geq4$, $\lambda=0$ is
neither an embedded eigenvalue nor a pole. Using~(\ref{1.2}),
\begin{equation}
\Sigma(t)=\int_0^t\ln\prod_{j=1}^m|\frac{x-\mu_j}{x-\overline{\mu}_j}|dx=0,\forall t\in\mathbb{R}.
\end{equation}
Hence, we conclude that
\begin{lemma}
\begin{equation}\label{1.21}
s(\lambda)\prod_{j=1}^m\frac{\lambda-\mu_j}{\lambda-\overline{\mu}_j}
=\exp\{i(a_0+a_1\lambda+\cdots+a_{n}\lambda^{n})\}
\frac{\overline{P}(\lambda)}{P(\lambda)},\hspace{2pt}\lambda\in\Lambda_{\frac{1}{2}}.
\end{equation}
\end{lemma}
Let us define
\begin{equation}
\sigma(\lambda):=\frac{1}{2\pi i}\log s(\lambda)
\end{equation}
and
\begin{equation}
g(\lambda):=i(a_0+a_1\lambda+\cdots+a_{n}\lambda^{n}).
\end{equation}
Functional analysis and~(\ref{111}) give
\begin{equation}\label{1.14}
\sigma'(\lambda)=\frac{1}{2\pi i}\frac{s'(\lambda)}{s(\lambda)}.
\end{equation}
In particular,~(\ref{1.21}) and~(\ref{1.14}) give
\begin{equation}\label{2.22}
\sigma'(\lambda)=\frac{1}{2\pi i}g'(\lambda)+\frac{1}{2\pi
i}\sum_{\mu_j}
\frac{1}{\lambda-\overline{\mu}_j}-\frac{1}{\lambda-\mu_j}
+\frac{1}{2\pi
i}\sum_{\lambda_j}\frac{1}{\lambda-\overline{\lambda}_j}-\frac{1}{\lambda-\lambda_j}
+Q_{\lambda}(\overline{\lambda_j})-Q_\lambda(\lambda_j),
\end{equation}
where
\begin{equation}
Q_\lambda(\lambda_j):=(\frac{1}{\lambda_j})(1+(\frac{\lambda}{\lambda_j})
+\cdots+(\frac{\lambda}{\lambda_j})^{n-1}),
\end{equation}
which is a polynomial in $\lambda$ provided $\{\lambda_j\}\neq0$.
$g'(\lambda)$ is also a polynomial of order no greater than $n-1$.

The main theorem of this paper is
\begin{theorem}
Let $n\geq4$ be even. Under the assumption that the set of closed
transversally reflected geodesics in ${\rm T}^\ast\Omega$ has measure zero, the resonance
counting function for $P$ in $\Lambda_{1}$,
\begin{equation}\label{1.27}
N_1(r)=\frac{(2\pi)^{-n}\omega_n{\rm
Vol}(\mathcal{O})}{\Gamma(n+1)}\{1+o(\frac{1}{r})\}r^n,\mbox{ as
}r\rightarrow\infty,
\end{equation}
where $\omega_n$ is the volume of the unit sphere in
$\mathbb{R}^n$.
\end{theorem}
The Weyl's asymptotics~(\ref{1.27}) is classical for the
eigenvalue counting problem in interior problems. We refer to
 Ivrii \cite{Ivrii} and Melrose \cite{Melrose2} for a discussion.

\section{A Proof}
Firstly, we need a satisfactory Poisson summation formula. We
start with a Birman-Krein formula. Such a formula is common in
scattering theory for all kinds of perturbation. For a black box
formalism setting, we refer to Christiansen
\cite{Christiansen,Christiansen2}.

\par
Defining the naturally regularized wave propagator,
\begin{equation}
u(t):=2\{\cos{t\sqrt{P}}-\cos{t}\sqrt{-\Delta}\}
\in\mathcal{D}'(\mathbb{R};\mathcal{J}_1(L^2(\Omega);L^2(\Omega))),
\end{equation}
where $\mathcal{J}_1(L^2(\Omega);L^2(\Omega))$ is the trace class
in $L^2$. $u(t)$ has a distributional trace. Let
$R^0(\lambda):=(-\Delta-\lambda^2)^{-1}$. By
spectral analysis, we have
\begin{eqnarray} {\rm
Tr}\{u(t)\}=\int_\mathbb{R}e^{it\lambda}2\lambda{\rm
Tr}\{R(\lambda)-R^0(\lambda)-R(-\lambda)+R^0(-\lambda)\}d\lambda
+2\sum_{\Im\mu_j>0}\cos(t\mu_j).\label{2.2}
\end{eqnarray}
By Birman-Krein theory, we have
\begin{equation}
\sigma'(\lambda)=2\lambda{\rm
Tr}\{R(\lambda)-R^0(\lambda)-R(-\lambda)+R^0(-\lambda)\}\in\mathcal{S}'(\mathbb{R}).
\end{equation}
On the other hand, we continue from~(\ref{2.22})
\begin{eqnarray}\nonumber
\sigma'(\lambda)&=&\frac{1}{2\pi i}g'(\lambda)+\frac{1}{2\pi
i}\sum_{\mu_j}
\frac{1}{\lambda-\overline{\mu}_j}-\frac{1}{\lambda-\mu_j}\\&&
+\frac{1}{2\pi
i}\sum_{\lambda_j}\frac{1}{\lambda-\overline{\lambda}_j}-\frac{1}{\lambda-\lambda_j}
+Q_\lambda(\overline{\lambda}_j)-Q_\lambda(\lambda_j)\in\mathcal{S}'(\mathbb{R}).
\end{eqnarray}
In our case, $g'(\lambda)$ and $Q_\lambda(\lambda_j)$ are at most
of order $n-1$. Hence,
\begin{eqnarray}
\partial_\lambda^{n}\sigma'(\lambda)=\frac{1}{2\pi
i}\sum_{\mu_j}
\frac{n!}{(\lambda-\overline{\mu}_j)^{n+1}}-\frac{n!}{(\lambda-\mu_j)^{n+1}}
+\frac{1}{2\pi
i}\sum_{\lambda_j}\frac{n!}{(\lambda-\overline{\lambda}_j)^{n+1}}-\frac{n!}{(\lambda-\lambda_j)^{n+1}}.
\end{eqnarray}
Therefore, for $t\neq0$,
\begin{eqnarray}\nonumber
t^n\int_\mathbb{R}e^{it\lambda}\sigma'(\lambda)d\lambda&=&\sum_{\mu_j}\frac{1}{2\pi
i}\int_\mathbb{R}e^{it\lambda}\{
\frac{n!}{(\lambda-\overline{\mu}_j)^{n+1}}-\frac{n!}{(\lambda-\mu_j)^{n+1}}\}d\lambda\\&&+\sum_{\lambda_j}\frac{1}{2\pi
i}\int_\mathbb{R}e^{it\lambda}\{\frac{n!}{(\lambda-\overline{\lambda}_j)^{n+1}}-\frac{n!}{(\lambda-\lambda_j)^{n+1}}\}
d\lambda. \label{2.5}
\end{eqnarray}
For $t\neq0$, we have
\begin{eqnarray}\nonumber
{\rm Tr}\{u(t)\}&=&\sum_{\mu_j}\frac{1}{2\pi
i}\int_\mathbb{R}e^{it\lambda}\{
\frac{1}{\lambda-\overline{\mu}_j}-\frac{1}{\lambda-\mu_j}\}d\lambda\\\nonumber
&&+\sum_{\lambda_j}\frac{1}{2\pi
i}\int_\mathbb{R}e^{it\lambda}\{\frac{1}{\lambda-\overline{\lambda}_j}-\frac{1}{\lambda-\lambda_j}\}d\lambda
\\&&+2\sum_{\Im\mu_j>0} \cos(t\mu_j).\label{2.7}
\end{eqnarray}
This cancelling $t\neq0$ technique is due to Zworski
\cite{Zworski2}. Most importantly, we have the following
well-known Poisson integral formula. See Lang \cite{Lang}.
\begin{lemma}
Let $f$ be a bounded holomorphic function over the closed upper
half plane. Defining
\begin{equation}
h_z(\zeta):=\frac{1}{2\pi
i}(\frac{1}{\zeta-z}-\frac{1}{\zeta-\overline{z}}),\mbox{ where
}\zeta,z\in\Lambda_{\frac{1}{2}},
\end{equation}
the following identity holds.
\begin{equation}
\int_{-\infty}^\infty f(\lambda)h_z(\lambda)d\lambda=f(z).
\end{equation}
\end{lemma}
Using this lemma,~(\ref{2.7}) becomes, when $t>0$,
\begin{eqnarray}
{\rm
Tr}\{u(t)\}=\sum_{\Im\lambda_j<0}e^{it\overline{\lambda}_j}
+\sum_{\Im\mu_j>0}e^{-it\mu_j},
\end{eqnarray}
which we rewrite as
\begin{eqnarray}\label{2.11}
{\rm Tr}\{u(t)\}=\sum_{\sigma_j}e^{it\sigma_j},\mbox{ in
}\mathcal{D}'(\mathbb{R}^+),
\end{eqnarray}
in which $\sigma_j$ are either $\overline{\lambda}_j$ or $-\mu_j$ such that $\mu_j^2$ are eigenvalues
of $P$. Alternatively, we can write
\begin{eqnarray}\label{2.10}
{\rm Tr}\{u(t)\}=\int_{\mathcal{Z}}
e^{it\sigma}dN_1(r(\sigma)),\mbox{ in }\mathcal{D}'(\mathbb{R}^+).
\end{eqnarray}
where
\begin{equation}
r(\sigma)\mbox{ is a norm of }\sigma\mbox{ as in }\mathbb{R}^2;
\end{equation}
\begin{equation}
\sigma\in\mathcal{Z}:=\{-\mu_1,\cdots,-\mu_m,\overline{\lambda}_1,\overline{\lambda}_2,\cdots\}.
\end{equation}
We see $N_1(r(\sigma))$ as a function of $\sigma\in\mathcal{Z}$.
For example, $N_1(|\sigma_j|)=j$ and the mapping $\sigma_j\mapsto
j$ is a well-defined function, so $N_1(|\sigma_j|)$ is seen as a
function of $\sigma_j$ not merely as a function of $|\sigma_j|$.
Moreover, the summation~(\ref{2.11}) is understood as the complex
abstract integration in Ash's book \cite[p.94,p.90(1)]{Ash} from
which we quote as Lemma 2.2 below. Hence, the integrand
$e^{it\sigma}$ and the integrator $dN_1(r(\sigma))$
in~(\ref{2.10}) are functions to variable $\sigma$.
\begin{lemma}\label{22}
If $f=(f(\alpha),\alpha\in\mathcal{Z})$ is a real- or  complex-valued
function on the arbitrary set $\mathcal{Z}$, and $\mu$ is counting
measure on subsets of $\mathcal{Z}$, then $\int_\mathcal{Z}
fd\mu=\sum_{\alpha}f(\alpha)$.
\end{lemma}
\par
For the left hand side of~(\ref{2.10}), we use the short time
asymptotic behavior of ${\rm Tr}\{u(t)\}$. It is well-known that
\begin{equation}
\label{2.17} {\rm Tr}\{u(t)\}\sim
a_0t^{-n}+a_1\delta^{(n-2)}(t)+\cdots,\hspace{2pt}\mbox{ as
}t\rightarrow0,
\end{equation}
where
\begin{equation}
a_0=(2\pi)^{-n}\omega_n{\rm
Vol}(\mathcal{O}),\hspace{2pt}a_1=\alpha_1{\rm
Vol}(\partial\mathcal{O}),\cdots,
\end{equation}
with $\omega_n$ as the volume of the unit sphere in
$\mathbb{R}^n$ and
$\omega_n=\frac{n\pi^{\frac{n}{2}}}{\Gamma(\frac{n}{2}+1)}$. The
constant $a_0$ here is derived in Ivrii \cite{Ivrii}.
Additionally, we refer to Branson and Gilkey \cite{Branson} for a
heat propagator version of this formula. Therefore,
\begin{equation}\label{2.16}
\int_\mathcal{Z}
e^{it\sigma}dN_1(r(\sigma))\sim(2\pi)^{-n}\omega_n{\rm Vol}(\mathcal{O}) t^{-n}+\cdots,\mbox{ as
}t\rightarrow0^+.
\end{equation}
We want to convert the information in $\mathbb{C}$ to be the one in
$\mathbb{R}$. Without loss of generality, we use the sup-norm from
now on. That is
\begin{equation}
N_1(r(\sigma)):=|\{\sigma_j||\Im\sigma_j|\leq
r(\sigma),|\Re\sigma_j|\leq r(\sigma)\}|.
\end{equation}
Accordingly, we set
\begin{equation}
N_\Re(r(\sigma)):=|\{\sigma_j|
r(\sigma)\geq|\Re\sigma_j|\geq|\Im\sigma_j|,|\Re\sigma_1|\leq|\Re\sigma_2|\leq\cdots\}|
\end{equation}
and
\begin{equation}
N_\Im(r(\sigma)):=|\{\sigma_j||\Re\sigma_j|<|\Im\sigma_j|\leq
r(\sigma),|\Im\sigma_{1'}|\leq|\Im\sigma_{2'}|\leq\cdots\}|.
\end{equation}
Let $\sigma=x+iy$. $N_\Re(r(\sigma))$ is a function of $x$ whenever $\Re\sigma\geq\Im\sigma$. In
particular,
\begin{equation}
N_\Re(r(\sigma))=N_\Re(\max\{|x|,|y|\})=N_\Re(|x|).
\end{equation}
Similarly, we use
$N_\Im(r(\sigma))=N_\Im(|y|)$ whenever $\Re\sigma<\Im\sigma$.
Accordingly, we decompose
\begin{eqnarray}\label{2.21}  \nonumber
\int_\mathcal{Z} e^{it\sigma}dN_1(r(\sigma))&=&\int_{|\Re\sigma|\geq|\Im\sigma|}
e^{it\sigma}dN_\Re(r(\sigma))+\int_{|\Re\sigma|<|\Im\sigma|} e^{it\sigma}dN_\Im(r(\sigma))\\
&=&\int_{|\Re\sigma|\geq|\Im\sigma|}
e^{it\sigma}dN_\Re(|x|)+\int_{|\Re\sigma|<|\Im\sigma|} e^{it\sigma}dN_\Im(|y|).
\end{eqnarray}

\begin{lemma}
Let $\sigma=x+iy\in\mathbb{C}$. We have
\begin{equation}\label{2.23}
\int_{|\Re\sigma|<|\Im\sigma|}
(e^{it\sigma}-e^{-ty})dN_\Im(|y|)\rightarrow0\mbox{ in
}\mathcal{D}'(\mathbb{R}^+),\mbox{ as }t\rightarrow0^+;
\end{equation}
 \begin{equation}
\int_{|\Re\sigma|\geq|\Im\sigma|}
(e^{it\sigma}-e^{itx})dN_\Re(|x|)\rightarrow0\mbox{ in
}\mathcal{D}'(\mathbb{R}^+),\mbox{ as }t\rightarrow0^+.
\end{equation}
\end{lemma}
\begin{proof}
Without of loss generality, we prove for all
$\varphi(t)\in\mathcal{C}^\infty_0(\mathbb{R}^+;[0,1])$. Since
$e^{-ty}\varphi(t)$ is nonnegative, Tonelli's theorem gives the
product measure
\begin{equation}
\int e^{-ty}\varphi(t)dN_\Im(|y|)\times dt=\int\{\int
e^{-ty}\varphi(t)dN_\Im(|y|)\}dt= \int\{\int
e^{-ty}\varphi(t)dt\}dN_\Im(|y|).
\end{equation}
Using Paley-Wiener's theorem, say, H\"{o}rmander \cite{Hormander},
\begin{eqnarray}\nonumber
&&|\int_{|\Re\sigma|<|\Im\sigma|}\int_0^\infty
e^{-ty}\varphi(t)dtdN_\Im(|y|)|\\\nonumber
&\leq&\int_{|\Re\sigma|<|\Im\sigma|}|\int_0^\infty
e^{-ty}\varphi(t)dt|dN_\Im(|y|))\\
&\leq&\int_{|\Re\sigma|<|\Im\sigma|} C_N(1+|y|)^{-N}e^{H(\Im(-iy))}dN_\Im(|y|),\forall N\in\mathbb{N},\nonumber\\
&=&\int_{|\Re\sigma|<|\Im\sigma|}
C_N(1+|y|)^{-N}e^{H(-y)}dN_\Im(|y|),\forall
N\in\mathbb{N},\label{2.28}
\end{eqnarray}
where $H$ is the related supporting function of $\varphi(t)$ and $C_N$ is a constant.
Without loss of generality, we assume $y>0$. Hence, the
integral~(\ref{2.28}) is convergent for large $N$ given the estimate of the growth order of $N_1(r)$ from~(\ref{119}) and~(\ref{110}). In particular,
\begin{equation}
e^{it\sigma}\varphi(t)\in L^1(dN_\Im(|y|)\times dt).
\end{equation}
Accordingly, Fubini's theorem gives
\begin{equation}
\int_0^\infty\int_{|\Re\sigma|<|\Im\sigma|}
(e^{it\sigma}-e^{-ty})\varphi(t)dN_\Im(|y|)dt=
\int_{|\Re\sigma|<|\Im\sigma|}\int_0^\infty
(e^{it\sigma}-e^{-ty})\varphi(t)dtdN_\Im(|y|).
\end{equation}
Now, let $\varphi(t)\in\mathcal{C}^\infty_0(\mathbb{R}^+;[0,1])$
such that $\int\varphi(t)dt=1$. We define
$\varphi_\gamma(t):=\frac{1}{\gamma}\varphi(t/\gamma)$, where
$\gamma>0$. To prove~(2.23), we show
\begin{eqnarray}\nonumber
&&|\int_{|\Re\sigma|<|\Im\sigma|}\int_0^\infty
(e^{it\sigma}-e^{-ty})\varphi_\gamma(t)dtdN_\Im(|y|)|,\mbox{
assuming } y>0,\\\nonumber
&\leq&\int_{|\Re\sigma|<|\Im\sigma|}\int_0^\infty
e^{-ty}|e^{itx}-1|\varphi_\gamma(t)dtdN_\Im(|y|)\\\nonumber
&=&\int_{|\Re\sigma|<|\Im\sigma|}\int_0^\infty
e^{-ty}|\sum_{j\geq1}\frac{(itx)^j}{j!}|\varphi_\gamma(t)dtdN_\Im(|y|)\\\nonumber
&\leq&\int_{|\Re\sigma|<|\Im\sigma|}\int_0^\infty
e^{-ty}\sum_{j\geq1}\frac{|tx|^j}{j!}\varphi_\gamma(t)dtdN_\Im(|y|),
\mbox{ given }|y|>|x|,\\\nonumber
&<&\int_{|\Re\sigma|<|\Im\sigma|}\int_0^\infty
e^{-ty}\sum_{j\geq1}\frac{(ty)^j}{j!}\varphi_\gamma(t)dtdN_\Im(|y|)\\\nonumber
&=&\int_{|\Re\sigma|<|\Im\sigma|}\int_0^\infty
e^{-ty}(e^{ty}-1)\varphi_\gamma(t)dtdN_\Im(|y|),\mbox{ letting
}t:=\gamma s,
\\
&=&\int_{|\Re\sigma|<|\Im\sigma|}\int_0^\infty e^{-\gamma s
y}(e^{\gamma s y}-1)\varphi(s)dsdN_\Im(|y|),
\end{eqnarray}
to which we apply the Lebesgue's dominated convergence theorem,
converging to $0$ as $\gamma\rightarrow0$.

The other identity follows similarly. $\Box$
\end{proof}
Using this lemma,
\begin{equation}
\int_\mathcal{Z}
e^{it\sigma}dN_1(r(\sigma))\rightarrow\int_{|\Re\sigma|\geq|\Im\sigma|}
e^{itx}dN_\Re(|x|)+\int_{|\Re\sigma|<|\Im\sigma|}
e^{-ty}dN_\Im(|y|),\mbox{ in }\mathcal{D}'(\mathbb{R}^+),\mbox{ as
}t\rightarrow0^+.
\end{equation}
For the first integral on $N_\Re(|x|)$, we set the change of
variable
\begin{equation}
\sigma:=x+iy\mapsto y+ix:=\sigma'.
\end{equation}
Let $\sigma':=x'+iy'$. We see
\begin{equation}
N_\Re(r(\sigma))=N_\Re(r(\sigma'))=N_\Re(r(y+ix))=N_\Re(\max\{|x|,|y|\})=N_\Re(|y'|).
\end{equation}
Since $\sigma'$ is a dummy variable,
\begin{equation}\label{2.34}
\int_\mathcal{Z}
e^{it\sigma}dN_1(r(\sigma))\rightarrow\int_{\mathbb{R^+}}
e^{-ty}d\{N_\Im(|y|)+N_\Re(|y|)\},\mbox{ in
}\mathcal{D}'(\mathbb{R}^+),\mbox{ as }t\rightarrow0^+.
\end{equation}
Please refer to the picture Figure 1 below for the counting
measure.

\begin{figure}\label{figure 1}\includegraphics[scale=0.6]{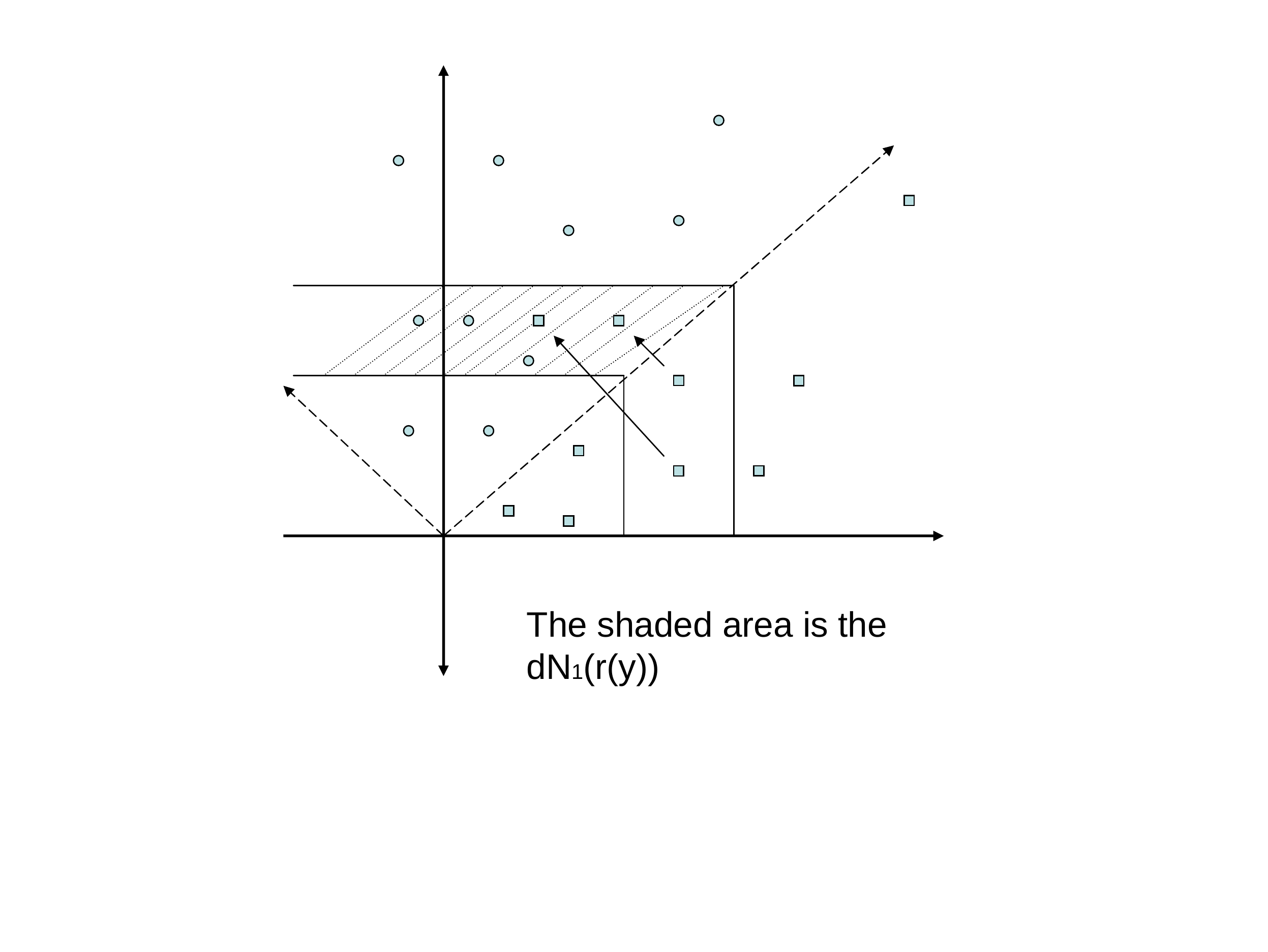}\caption{the counting measure in active}\end{figure}
Now we recall the Karamata's Tauberian theorem. As stated in
Taylor's book \cite{Taylor},
\begin{proposition}
If $\mu$ is a positive measure on $[0,\infty)$,
$\alpha\in(0,\infty)$, then
\begin{equation}\label{2.16}
\int_0^\infty e^{-t\lambda}d\mu(\lambda)\sim at^{-\alpha},
\hspace{2pt}t\searrow0,
\end{equation}
implies
\begin{equation}
\int_0^r d\mu(\lambda)\sim br^{\alpha},
\hspace{2pt}r\nearrow\infty, \mbox{ with }
b=\frac{a}{\Gamma(\alpha+1)}.
\end{equation}
\end{proposition}
In our case, $a=(2\pi)^{-n}\omega_n{\rm Vol}(\mathcal{O})$ and
$\alpha=n$. By construction, $N_\Im(|y|)+N_\Re(|y|)$ is a
nondecreasing function of $|y|$, so~(\ref{2.17}),~(\ref{2.34})
and~(\ref{2.16}) yield
\begin{equation}
N_1(r)=\frac{(2\pi)^{-n}\omega_n{\rm
Vol}(\mathcal{O})}{\Gamma(n+1)}\{1+o(\frac{1}{r})\}r^n,\mbox{ as
}r\rightarrow\infty.
\end{equation}
This proves the main theorem.

\begin{acknowledgement}
The author wants to thank many mathematicians for the critique on the previous version of this work and Prof.
Tzy-Wei Hwang for reading through the manuscript.
\end{acknowledgement}

\end{document}